\documentclass[12pt]{article}
\usepackage{amsmath,amssymb,amsfonts}  
\usepackage{amsthm}
\usepackage{graphicx,psfrag,epsf}
\usepackage{enumerate}
\usepackage{natbib}
\usepackage{url} 
\usepackage{microtype}
\usepackage{xspace}
\usepackage{tikz} 
\usetikzlibrary{arrows,decorations.pathmorphing,backgrounds,positioning,fit,petri}
\newcommand{\blind}{0}

\usepackage{graphicx}
\usepackage{amsmath}      
\usepackage{xspace}
\usepackage{microtype}

\usepackage{color}
\usepackage{dsfont}
\usepackage{algorithm}
\usepackage{algpseudocode}
\DeclareMathAlphabet{\mathantt}{OT1}{antt}{m}{n}

\graphicspath{{./pictures/},{./}}

\newcommand*{\dist}{\ensuremath{\rho}\xspace} 

\newcommand*{\rand}[1]{\ensuremath{\mathbf{#1}}\xspace}

\newcommand*{\funcf}{\ensuremath{\mathantt{f}}\xspace}

\newcommand*{\acc}{\ensuremath{\mathantt{a}}\xspace} 
\newcommand*{\pcc}{\ensuremath{\mathantt{p}}\xspace}

\newcommand*{\prob}[1]{\ensuremath{\mathbb{P}\left(#1\right)}\xspace}

\newcommand{\ind}[1]{\raisebox{-0.3mm}{\scalebox{1.2}{\ensuremath{\mathds{1}}}}\ensuremath{\{ #1\}}}

\newcommand*{\sub}{\ensuremath{\subseteq}}

\newcommand*{\limi}[1]{\ensuremath{\lim\limits_{#1\rightarrow \infty}}} 
 
\newcommand*{\limii}[2]{\ensuremath{\ \overset{#1\rightarrow #2}{\longrightarrow}\ }}

\newcommand*{\states}{\ensuremath{\mathcal{S}}\xspace}
\newcommand*{\trans}{\ensuremath{\kappa}\xspace}

\newcommand*{\mcal}[1]{\ensuremath{\mathcal{#1}}\xspace}

\newcommand*{\todo}[1]{}

\makeatletter
\newcommand*{\defeq}{\mathrel{\rlap{\raisebox{0.3ex}{$\m@th\cdot$}}\raisebox{-0.3ex}{$\m@th\cdot$}}=}
\makeatother
{\bf}{\rm}
{\bf}{\rm}
{\bf}{\rm}
{\bf}{\rm}
\newtheorem{theorem}{Theorem}{\bf}{\rm}
{\bf}{\rm}
{\bf}{\rm}
{\bf}{\rm}
{\bf}{\rm}
\begin{document}

\def\spacingset#1{\renewcommand{\baselinestretch}%
{#1}\small\normalsize} \spacingset{1}


\if0\blind
{  
  \title{\bf Markov Chain Monte Carlo on Finite State Spaces}
  \author{Tobias Siems\\
    Department of Mathematics and Computer Science\\
			University of Greifswald\\\\
	\underline{\textbf{Accepted for publication with minor corrections in}}\\
	\underline{\textbf{The Mathematical Gazette issue July 2020}}}
  \maketitle
} \fi

\if1\blind
{
  \begin{center}
    {\LARGE\bf Title}
\end{center}
} \fi

\spacingset{1.45} 
\sloppy
\begin{abstract}
We elaborate the idea behind Markov chain Monte Carlo (MCMC) methods in a mathematically coherent, yet simple and understandable way.
To this end, we proof a pivotal convergence theorem for finite Markov chains and a minimal version of the Perron-Frobenius theorem.
Subsequently, we briefly discuss two fundamental MCMC methods, the Gibbs and Metropolis-Hastings sampler.
Only very basic knowledge about matrices, convergence of real sequences and probability theory is required.
\end{abstract}

\noindent%
\vfill

\newpage
\section{Introduction}
\label{chap:intro}


MCMC techniques aim at drawing samples from prespecified distributions. 
They do so in an indirect, approximate fashion through Markov chains.
This is important since the distributions deployed in practice are often too complex to be dealt with directly or even unavailable in closed form.

There exists a tremendous number of scientific articles and books about MCMC.
See for example \cite{bishop2014} for a vivid and more comprehensive introduction without mathematical proofs.

Let $\pi$ be a distribution over a finite state space $\states$ and $\pi_s\in[0,1]$ the probability of state $s\in\states$ under $\pi$, whereby $\sum_{s\in\states}\pi_s=1$.
By virtue of the strong law of large numbers, independent samples from $\pi$ (so-called \emph{i.i.d. samples}) can be used universally to approximate expectations w.r.t. $\pi$.
Thus, for a set of such samples, say $x_1,...,x_m\in\states$, and an arbitrary function $\funcf:\states\rightarrow \mathbb{R}^\ell$, we get that $\frac{1}{m}\sum_{i=1}^m \funcf(x_i)\approx\sum_{s\in\states}\pi_s\funcf(s)$.
This simple recipe poses one of the most powerful tools in statistics.

An example for $\funcf$ is the indicator function $\ind{s\in A}$ for events $A\sub\states$.
It is one if the condition in the brackets is true and zero otherwise.
Its expectation yields the probability of $A$.


A \emph{Markov chain} over $\states$ is defined through an (arbitrary) initial state $s_0\in \states$ and a \emph{transition kernel} $\trans$.
$\trans$ is a non-negative function over $\states\times\states$ such that $\sum_{s\in \states}\trans_{z s}=1$ for all $z\in\states$.
$\trans_{z}$ can be interpreted as a conditional distribution.

The Markov chain starts in state $s_0$ and evolves according to $\trans$ in an iterative fashion:
The distribution of the first link in the chain is $\trans_{s_0 }$ and given the first link, say $s_1$, the distribution of the second link is $\trans_{s_1 }$ and so forth. 
This results in a potentially infinite sequence of random variables $\rand X=(\rand X_0, \rand X_1,\rand X_2,...)$, whereby $\rand X_n$ represents the $n$-th link in the chain.
Consequently, we get that $\prob{\rand X_n=s\mid \rand X_{n-1}=z}=\trans_{zs}$ for $n>0$ and $\prob{\rand X_0=s_0}=1$.

Later on, we will deal with the (unconditional) distributions of the $n$-th links.
To this end, w.l.o.g. we assume that $\states=\{1,...,k\}$.
Therewith, we describe $\trans$ as a matrix in $[0,1]^{k\times k}$ and $\pi$ as a column vector in $[0,1]^k$.
Any quadratic matrix with non-negative entries and rows that sum to one is called \emph{stochastic matrix}.
Therewith, $\prob{\rand X_n=s}=(\trans^n)_{s_0s}$, i.e. the $n$-fold matrix product of $\trans$ evaluated at $s_0$ and $s$.
A further generalization is to set the law of $\rand X_0$ to an arbitrary distribution $\pi$. 
This yields $\prob{\rand X_n=s}=(\pi^t\trans^n)_{s}$, whereby the $t$ indicates transposition.

We say that a distribution $\pi$ is an \emph{invariant distribution} of $\trans$ if $\pi^t\trans=\pi^t$. 
Thus, transitioning according to $\trans$ doesn't affect $\pi$.
Once a link in a Markov chain with transition kernel $\trans$ follows the law $\pi$, all subsequent links will do likewise.
In this case, the chain is considered to be in \emph{equilibrium}.
Equilibrium can be enforced by setting the distribution of $\rand X_0$ to $\pi$, but it may also be reached (approximately) in the long run through convergence.

The foundation of MCMC sampling is that under some circumstances the distributions of the $n$-th links of a Markov chain converge towards an invariant distribution regardless of the initial state. 
Thus, by simulating such a chain until equilibrium is reached sufficiently, we may obtain an approximate sample of this invariant distribution.
On these grounds, MCMC methods provide schemes to build Markov chains with a predefined unique invariant distribution.

The Gibbs sampler \citep{Geman1984} builds a Markov chain with invariant distribution $\pi$ by decomposing $\pi$ into simpler conditional versions.
This facilitates sampling of complex joint distributions, but is somewhat restricted in its ability to explore the state space.

The well-known Metropolis-Hastings algorithm \citep{metropolis1953, hastings1970} is capable of incorporating user defined proposal distributions, which
enables the exploration of the state space in any desired fashion.
It further facilitates the application of complex statistical models to observed data.

\subsection{Convergence towards and existence of invariant distributions}

At first, we consider the convergence of the distributions of the $n$-th links of certain Markov chains.
This convergence forms the very basis of MCMC sampling.
At second, we provide a version of the Perron-Frobenius theorem, which gives further insights into the existence of invariant distributions.
The following two fundamental properties of $\trans$ are imposed by these theorems.

$\trans$ is called \emph{irreducible} if for every $z,s\in\states$ there exists an $n>0$ such that $(\trans^n)_{zs}>0$.
Thus, regardless of the state the Markov chain starts in, every state can eventually be reached with positive probability.

$\trans$ is called \emph{aperiodic} if there exists an $N$ such that $(\trans^n)_{zs}>0$ for all $n>N$ and all $z,s\in\states$.
Thus, regardless of the state the Markov chain starts in, in the long run, it can always reach any state immediately with positive probability.

The name irreducible suggests that the Markov chain does not divide $\states$ into separate, mutually inaccessible classes.
In turn, aperiodicity excludes the case where parts of $\states$ can only be reached in a periodic fashion,
for example only through either an even or odd number of transitions.
Aperiodicity implies irreducibility, but not the other way around.

It is obvious that a periodic behavior may impede convergence of the distributions of the $n$-th links.
For a general convergence theorem, aperiodicity is thus a necessary condition.
We will now see that it is also sufficient.
The following theorem is a simplified version of a convergence theorem for Markov chains over countable state spaces provided inter alia by \cite{konig2005stochastische}.

\begin{theorem}
	\label{lem:invdisc}
	For an aperiodic stochastic matrix $\trans$ with invariant distribution $\pi$, we get that $\limi{n}(\trans^n)_{zs}= \pi_s$ for all $z,s\in\states$.
\end{theorem}
\begin{proof}
	Assume that $\rand X=(\rand X_0, \rand X_1, ...)$ is a Markov chain with transition kernel $\trans$ that starts with an arbitrary but fixed state $s_0$.
	Furthermore, consider the Markov chain $\rand{Z}=(\rand Z_0, \rand Z_1,...)$ with transition kernel $\trans$ and initial distribution $\pi$, i.e. $\prob{\rand Z_0=s}=\pi_s$ for all $s\in\states$.
	$\rand X$ and $\rand Z$ are supposed to  be independent of each other.

	Let $\rand T$ be the random variable that represents the first $n$ where $\rand X$ and $\rand Z$ equal, i.e. $\rand T=\min\{n\mid \rand X_n=\rand Z_n\}$.
	We want to show that $\rand T$ is finite with probability one.
	Due to the aperiodicity of $\trans$, we may choose an $N$ such that $\trans^N$ has solely positive entries. 
	Let $\epsilon>0$ be the smallest entry of $\trans^N$ and consider that
	\begin{align*}
	&\prob{nN< \rand T}=\prob{\rand X_{i}\not=\rand Z_{i}\text{ for all } i\leq nN}\leq	\prob{\rand X_{i\cdot N}\not=\rand Z_{i\cdot N}\text{ for all } i\leq n}\\
	&=\sum_{z_0\not=s_0}\pi_{z_0}\sum_{s_1\in\states}\trans^N_{s_0s_1}\sum_{z_1\not=s_1}\trans^N_{z_0z_1}...
	\sum_{s_n\in\states}\trans^N_{s_{n-1}s_n}\underbrace{\sum_{z_n\not=s_n}\trans^N_{z_{n-1}z_n}}_{\leq 1-\epsilon}\leq (1-\epsilon)^n\overset{n\rightarrow \infty}{\longrightarrow}\ 0
	\end{align*}

	Define further $\rand Y_n=\ind{n\leq \rand T}\rand X_n+\ind{n>\rand T}\rand Z_n$.
	The Markov chain $\rand Y=(\rand Y_0, \rand Y_1, ...)$ starts with copying $\rand X$ and switches to $\rand Z$ as soon as both equal the first time.
	We are interested in the distribution of $\rand Y$.
	To this end, consider an arbitrary path $s_{1:n}\in\states^{n}$, define $\pcc_{s_{j:i}}=\prod_{\ell=j}^{i} \trans_{s_{\ell-1} s_\ell}$ for $j,i=1,...,n$ and observe that
	\begin{align*}
	&\prob{\rand Y_{0:n}=s_{0:n}}=\prob{\rand Y_{0:n}=s_{0:n}, n< \rand T}+\sum_{\ell=0}^n\prob{\rand Y_{0:n}=s_{0:n}, \rand T=\ell}\\
	&=\pcc_{s_{1:n}}\prob{\rand Z_{0:n}\neq s_{0:n}}
	+\sum_{\ell=0}^n\pcc_{s_{1:\ell}}\prob{\rand Z_\ell=s_\ell, \rand Z_{0:\ell-1}\neq s_{0:\ell-1}}\pcc_{s_{\ell+1:n}}=\pcc_{s_{1:n}}
	\end{align*}
	whereby we used that $\prob{\rand Z_{0:n}\neq s_{0:n}}+\sum_{\ell=0}^n\prob{\rand Z_\ell=s_\ell, \rand Z_{0:\ell-1}\neq s_{0:\ell-1}}=1$.
	This shows that $\rand Y$ is a Markov chain with transition kernel $\trans$ and initial state $s_0$.

	Altogether, we may state that 
	\begin{align*}
	(\trans^n)_{s_0s}&=\prob{\rand Y_n=s}=\prob{\rand Y_n=s, n\leq \rand  T}+\prob{\rand Y_n=s, n>\rand T }\\
	\pi_s&=\prob{\rand Z_n=s}=\prob{\rand Z_n=s, n\leq \rand T}+\prob{\rand Y_n=s, n>\rand T}
	\end{align*} 
	and
	$|(\trans^n)_{s_0s}-\pi_s|=|\prob{\rand Y_n=s, n\leq \rand T }- \prob{\rand Z_n=s, n\leq \rand T}|\limii{n}{\infty} 0$ for all $s_0\in\states$.
\end{proof}


Given a distribution $\pi$, MCMC methods seek an aperiodic transition kernel $\trans$ with invariant distribution $\pi$.
Thus, it is possible to sample approximately from $\pi$ by simulating a Markov chain with transition kernel $\trans$ until equilibrium is reached to a sufficient extend.
The last sample within this chain is then taken as a single approximate sample from $\pi$.
In particular, this procedure is independent of the state the Markov chain has started in.
The pace with which equilibrium is approached is referred to as the \emph{mixing time}.

Now, we consider a version of the well-known Perron-Frobenius theorem \citep{frobenius1912matrizen}.
It is usually stated in a more general context and corresponding proofs can be fairly complicated.
In turn, we provide our own convenient proof based on simple arithmetics and matrix algebra.

\begin{theorem}[Perron-Frobenius Theorem]
\label{lem:irredmat}
An irreducible transition kernel $\trans$ has a unique invariant distribution $\pi$.
\end{theorem}
\begin{proof}
Since any stochastic matrix has a right eigenvector with corresponding eigenvalue 1, it also has such a left eigenvector.
In particular, any such left 1-eigenvector exhibits non-zero elements.
Let $x$ be a left 1-eigenvector of \trans.
If $x$ has only non-negative or non-positive entries, we can immediately derive an invariant distribution $\pi$ of $\trans$ through normalizing $x$, i.e. $\pi=x\big\slash \sum_{s\in\states} x_s$.

Assume now that $x$ exhibits positive entries for $s\in \bar N$ and negative entries for $s\in N$.
The following applies
\begin{align}
\nonumber
&\sum_{z\in N}x_z \trans_{zs}+\sum_{z\in \bar N}x_z \trans_{zs}=x_s
\ \ \Rightarrow\ \ 
\sum_{z\in N}x_z \sum_{s\in\bar N}\trans_{zs}+\sum_{z\in \bar N}x_z \sum_{s\in \bar N}\trans_{zs}=\sum_{s\in \bar N}x_s\\
\label{eq:perr}
\Leftrightarrow\ \ &\underbrace{\sum_{z\in N}x_z \sum_{s\in\bar N}\trans_{zs}}_{\leq 0}
=\sum_{s\in \bar N}x_s-\sum_{z\in \bar N}x_z\sum_{s\in \bar N}\trans_{zs}
=\sum_{z\in \bar N}x_z\left(1-\sum_{s\in \bar N}\trans_{zs}\right)
=\underbrace{\sum_{z\in \bar N}x_z \sum_{s\not \in \bar N}\trans_{zs}}_{\geq 0}
\end{align}
Hence, the l.h.s and r.h.s. of (\ref{eq:perr}) have to be zero, which implies that $\trans_{zs}=0$ for all $z\in \bar N$ and $s\not\in \bar N$.
Consequently, $(\trans^n)_{zs}=0$ for all $z\in \bar N$, $s\not\in \bar N$ and $n\in\mathbb{N}$.

Since the existence of positive and negative entries implies reducibility, we conclude that irreducibility implies that any left 1-eigenvector has either solely non-positive or non-negative entries.
Thus, an irreducible transition kernel $\trans$ exhibits at least one invariant distribution $\pi$.

Finally, assume that there is a second invariant distribution $\pi'\neq \pi$.
In order to represent a distribution, not all components of $\pi$ can be either larger or smaller than the components of $\pi'$.
Thus, $\pi-\pi'$ must have positive and negative entries.
However, $\pi-\pi'$ is a left 1-eigenvector of $\trans$ and thus, $\trans$ can't be irreducible, which contradicts the existence of $\pi'$.
\end{proof}

The Perron-Frobenius theorem shows that invariant distributions can certainly be found for an abundance of stochastic matrices, especially for aperiodic ones.
In the context of MCMC, it is, however, only a nice-to-have result and not utterly necessary.
In fact, there is great freedom in choosing aperiodic transition kernels that exhibit a prespecified invariant distribution and each MCMC method provides its very own approach to do so.

\subsection{The Gibbs Sampler}

The Gibbs sampler \citep{Geman1984} is a primal MCMC sampling algorithm which is based on a decomposition of the objective distribution $\pi$ into conditional versions.
It is mainly used to sample from the joint distribution of a set of random variables.
Thereby each step may involve sampling from a single random variable given the remaining random variables conditioned on the last sample.

Assume that $\states=\mcal{V}^m$, with a finite space $\mcal{V}$ and let $\pi$ be a distribution over $\states$.
We proceed in a step wise manner.
In step $i$, given the sample $z\in\states$ from the last step $i-1$, we choose $j=(i\mod m)+1$ and sample from the transition kernel $\trans^j_z$ which is defined through
\begin{align}
\label{eq:gibbs}
\trans^j_{zs}=\frac{\ind{z_\ell=s_\ell \text{ for } \ell\neq j}\pi_s}{\sum_{s'\in\states}\ind{z_\ell=s'_\ell \text{ for } \ell\neq j}\pi_{s'}}
\end{align}
$\trans^j_{zs}$ is the conditional version of $\pi_s$ conditioned on $s_\ell=z_\ell$ for all $\ell$ except the $j$-th one.

We get that $\pi$ is an invariant distribution of $\trans^j$ since $\sum_{z\in\states}\pi_z\trans^j_{zs}=\sum_{z\in\states}\pi_{s}\trans^j_{sz}=\pi_{s}$.
However, a single Gibbs step is generally not aperiodic since it manipulates only one single coordinate.
This may be fixed by considering $m$ subsequent Gibbs steps as one transition within the Markov chain of the ensuing MCMC algorithm.

Unfortunately, the possible transitions are completely determined by $\pi$.
For example, assume that $\states=\{0,1\}^2$ and $\pi_{(0,0)}=\pi_{(1,1)}=0$.
In this case, we can never get from $(1,0)$ to $(0,1)$ and thus, the chain is not irreducible.
This poses one of several reasons why it can be important to be able to traverse the state space in a more arbitrary fashion.

A famous application of the Gibbs sampler is the Ising model \citep{Ising1925}.
There, $\states$ consists of the positive or negative states of the grid points of a finite grid, whereby independence is induced by spatial separation.
This yields very simple sampling steps, each conducted on a single grid point given all the other, but essentially only its neighboring grid points.
\cite{higdon1998} provides very vivid and more sophisticated treatments of the Ising model through Gibbs sampling.

\subsection{The Metropolis-Hastings algorithm}
\label{chap:mh}

The well-known Metropolis-Hastings algorithm is an MCMC sampler that traverses through the state space by means of a user defined proposal distribution.
Inherent for this sampler is that each proposed state undergoes an accept-reject step which decides whether the proposed state or the previous sample is chosen to be the next sample.
This step alone secures the required invariance and thus, gives the user great freedom in exploring the state space.
A primal version was first published in \cite{metropolis1953} and then extended in \cite{hastings1970}.

The user provides a transition kernel $\dist$ over $\states$ which is referred to as the \emph{proposal}.
In step $i$, given the previous sample $z$, a new state $s$ is proposed according to $\dist_z$ and accepted with probability 
\begin{align}
\label{eq:acc}
\acc_{zs}=\min\left\{1,\frac{\pi_{s}\dist_{sz}}{\pi_z\dist_{zs}}\right\}
\end{align}
whereby we agree that dividing by zero yields $\infty$.
The new sample is than either $s$ if accepted or $z$ otherwise.

The invariance can be shown by means of the following condition.
A transition kernel $\trans$ preserves the \emph{detailed balance} w.r.t. $\pi$ if $\pi_s\trans_{sz}=\pi_{z}\trans_{zs}$ for all $s,z\in\states$. 
Detailed balance implies invariance since $\sum_{z\in\states}\pi_z\trans_{zs}=\pi_{s}\sum_{z\in\states}\trans_{sz}=\pi_{s}$, but the opposite implication does not hold in general.

A Markov chain that preserves the detailed balance w.r.t. another distribution $\pi$ is said to be time-reversible in equilibrium or simply called \emph{reversible}.
That is, if $\rand X_n$ follows the law $\pi$, we get $\prob{\rand X_{n+1}=s\mid \rand X_{n}=z}=\prob{\rand X_n=z\mid \rand X_{n+1}=s}$.

The Metropolis-Hastings algorithm describes a reversible Markov chain.
This is trivial for transitions from $z$ to $s$ with $s=z$. For $s\neq z$, consider that
\begin{align*}
\pi_z \dist_{zs} \acc_{ zs}=\min\Big\{\pi_z\dist_{zs},\pi_{s}\dist_{sz}\Big\}=\pi_{s}  \dist_{sz}\acc_{sz}
\end{align*}

The aperiodicity of the resulting Markov chain has to be met jointly by the acceptance probability and the proposal.
Equation (\ref{eq:acc}) shows that a non-trivial transition from $z$ to $s$ can only take place if the corresponding backward transition is also feasible, i.e. if $\pi_{s}\dist_{sz}>0$.
In the most extreme case, this could mean that we apply an irreducible and aperiodic proposal, but the resulting Metropolis-Hastings Markov chain will never move away from the initial state and is thus not irreducible.

Apart from the ability to traverse the state space in a user defined fashion, a very important advantage of this sampling algorithm is that normalization constants w.r.t. $\pi$ cancel out in (\ref{eq:acc}).
Thus, we may sample from conditional versions of $\pi$ by using $\pi$ directly in the acceptance probability.
This enables convenient data-driven inference approaches by means of carefully designed and complex statistical models, without the need to compute cumbersome marginalizations over unobserved states.

\bibliographystyle{chicago}      
\bibliography{../bibtex}   


\end{document}